\documentclass{article} 
\usepackage{amssymb,amsmath,amsthm}
\usepackage{amscd}
\usepackage{hyperref}
\usepackage{epsfig}
\usepackage{amsfonts}
\usepackage{amssymb}
\usepackage{amsmath,enumerate}
\usepackage{commath}
\usepackage{euscript}
\usepackage{enumitem}
\usepackage{amscd}
\usepackage{xcolor}
\usepackage[ruled,vlined]{algorithm2e}
\usepackage[numbers,sort&compress]{natbib}

\newtheorem{theorem}{Theorem}[section]
\newtheorem{remark}{Remark}[section]

\newtheorem{definition}{Definition}[section]

\newtheorem{proposition}{Proposition}[section]
\newtheorem{example}{Example}[section]

\begin{document}
	\title{Variational Analysis of Generalized Games over Banach spaces}
	\author{
		Asrifa Sultana\footnotemark[1] \footnotemark[2] , Shivani Valecha\footnotemark[2] }
	\date{ }
	\maketitle
	\def\thefootnote{\fnsymbol{footnote}}
	
	\footnotetext[1]{ Corresponding author. e-mail- {\tt asrifa@iitbhilai.ac.in}}
	\noindent
	\footnotetext[2]{Department of Mathematics, Indian Institute of Technology Bhilai (491001), India.
	}
	
	
		
	\vspace{-1em}
	\begin{abstract}\noindent
		We study generalized games defined over Banach spaces 
		using variational analysis. To reformulate generalized games as quasi-variational inequality problems, we will first form a suitable principal operator and study some significant properties of this operator. Then, we deduce the sufficient conditions under which an equilibrium for the generalized game can be obtained by solving a quasi-variational inequality. Based on this variational reformulation, we derive the existence of equilibrium for generalized games with non-ordered (that is, associated weak preference relations need not be complete and transitive) and mid-point continuous preference maps. 
	\end{abstract}
	\section{Introduction}\label{intro}
	Arrow-Debreu \cite{debreu} initiated the notion of generalized Nash equilibrium problems (GNEP) consisting of finitely many players with their strategy sets in Euclidean spaces. Recently, several authors have studied these problems on infinite dimensional spaces motivated by dynamic games defined over a time interval and stochastic games consisting of infinite number of states of nature (see \cite{faraci,ausseltime} and references therein). Consider a set of players $\Lambda=\{1,2,\cdots N\}$ and finite a collection of Banach spaces $\{X_\nu|\,\nu\in \Lambda\}\}$. Suppose each $\nu\in \Lambda$ regulates a strategy variable $x_\nu\in C_\nu$ where $C_\nu\subseteq X_\nu$. Assume that, 
	\begin{equation}\label{X}
		C_{-\nu}=\prod_{\mu\in \Lambda\setminus\{\nu\}} C_\mu, C= \prod_{\nu\in \Lambda}C_\nu, X_{-\nu}=\prod_{\mu\in \Lambda\setminus\{\nu\}} X_\mu~\text{and}~ X=\prod_{\nu\in \Lambda}X_\nu.
	\end{equation} 
	Then, any vector $x_{-\nu}\in C_{-\nu}$ consists of strategies of players other then $\nu$ and we have $x=(x_\nu,x_{-\nu})\in C$. Suppose $u_\nu:X\rightarrow \mathbb{R}$ is objective function (or numerical representation of preferences) of player $\nu$.
	For a given strategy vector $x_{-\nu}$ of rival players, the player $\nu$ intends to find $x_\nu$ in the feasible strategy set $K_\nu(x_\nu,x_{-\nu})\subseteq C_\nu$ such that,
	\begin{equation}\label{GNEP2}
		u_\nu(x_\nu,x_{-\nu})=\max_{z_\nu\in K_\nu(x)} u_\nu(z_\nu,x_{-\nu}).
	\end{equation}
	Suppose $Sol_\nu(x_{-\nu})$ consists of all the vectors $x_\nu$ which satisfy (\ref{GNEP2}). Then, any vector $\bar x\in X$ is known as equilibrium for GNEP $\Upsilon=(X_\nu,K_\nu,u_\nu)_{\nu\in \Lambda}$ if $\bar x_\nu\in Sol_\nu(\bar x_{-\nu})$ for each $\nu\in \Lambda$ \cite{debreu,faccsurvey}.
	
	A preference relation of any individual is a binary relation which represents his choice in the set
	of available options. The preference relation of any individual satisfy the completeness and transitivity property if he is able to compare any pair of available options. Any complete, transitive and continuous preference relation defined on connected and separable topological space admits a numerical representation (see \cite[Theorem I]{decision_processes}). Hence, one can obtain the above well known version of GNEP with real-valued objective functions if preferences of players satisfy these properties. It is worth mentioning that for numerical representation of preference relations the completeness and transitivity properties are necessary (see \cite[Proposition 1.3]{krepsweak}). The fact that preference relations  are not always complete and transitive in real-world scenarios (see for e.g. \cite{aumann}) motivated Shafer-Sonnenschein \cite{shafer} to study generalized games in which 
	players' preference maps are non-ordered, that is, there are no restrictions like completeness and transitivity of players' preference relations. Thereafter, Yannelis- Prabhakar \cite{yannelis_infinite} extended the generalized games studied in \cite{shafer} to the case where strategy spaces are contained in topological vector spaces.  
	
	Let us recall the notion of generalized games studied in \cite{shafer,yannelis_infinite} which extends the GNEP $\Upsilon=(X_\nu,K_\nu,u_\nu)_{\nu\in \Lambda}$ to the case of non-ordered preferences. Suppose the strategy set $C_\nu$ of any player $\nu\in \Lambda$ is contained in Banach space $X_\nu$ same as GNEP $\Upsilon$ and $C_{-\nu},C,X_{-\nu},X$ are defined as (\ref{X}).  
	Let us denote generalized game (or abstract economy) by $\Gamma=(X_\nu,K_\nu,P_\nu)_{\nu\in \Lambda}$ where $P_\nu:X\rightrightarrows X_\nu$ and $K_\nu:C\rightrightarrows C_\nu$ are non-ordered preference and feasible strategy maps of player $\nu$, respectively. Suppose $\mathbb{S}(\Gamma)$ indicates the set of equilibrium points for generalized game $\Gamma$. Then by following \cite{shafer,yannelis_infinite}, we say $\bar x\in \mathbb{S}(\Gamma)$ if,
	\begin{equation}\label{GNEPshafer}
		\bar x_\nu\in K_\nu(\bar x)~\text{and}~ P_\nu(\bar x)\cap K_\nu(\bar x)=\emptyset~\text{for every}~ \nu\in \Lambda.
	\end{equation}
	
	Recently, the authors studied economic equilibrium problems with complete, transitive and continuous preference relations in \cite{milasipref} using tools of variational analysis. Thereafter, the maximization problem for incomplete and non-transitive preferences has been studied in \cite{milasipref2021,donato} by reformulating it as a variational inequality problem. In these articles, the principal operator of associated variational inequality is formed by using normal cones to preferred sets.  This motivated Beuno et. al. \cite{cotrina} to study existence of equilibrium for generalized games by employing variational approach. Further, Sultana-Valecha deduced the existence of equilibrium and projected solutions for generalized games with not necessarily bounded strategy maps \cite{shivani1} and non-self constraint maps \cite{shivani2}, respectively, by reformulating these games into appropriate quasi-variational inequality problems. However, the mentioned articles \cite{milasipref2021, donato,cotrina,shivani1,shivani2} are restricted to finite dimensional strategy spaces. We observe that the variational approach used in all these articles can not be administered on the generalized games $\Gamma=(X_\nu,K_\nu,P_\nu)_{\nu\in \Lambda}$ defined over infinite dimensional Banach spaces $X_\nu$ (see Remark \ref{infinite_remark}). 
	
	Our aim is to study the generalized game $\Gamma=(X_\nu,K_\nu,P_\nu)_{\nu\in \Lambda}$ defined over Banach spaces with non-ordered (that is, associated weak preference relations need not be complete and transitive) and mid-point continuous preference maps, by employing variational inequality theory. To reformulate generalized games as quasi-variational inequality problems, we will first form a suitable principal operator and study some significant properties of this operator. Then, we deduce the sufficient conditions under which an equilibrium for the generalized game can be obtained by solving an associated quasi-variational inequality. Based on this variational reformulation, we derive the existence of equilibrium for the proposed generalized games. 
	
	\section{Notations and Definitions}
	Suppose $X$ is a Banach space with topological dual $X^*$ and duality pairing $\langle\cdot,\cdot\rangle$. We denote unit ball in $X$ by $B_X=\{y\in X|\,\norm{y}\leq 1\}$ and unit ball in $X^*$ by $B_X^*=\{y\in X^*|\,\norm{y}_*\leq 1\}$, where $\norm{\cdot}_*$ is usually defined as $\norm{y}_*=\sup_{\norm{u}\leq 1}\langle y^*,u\rangle$ (see \cite{aliprantis}). For any $C\subseteq X$, we assume $co(C)$ and $\overline{C}$ denote convex hull and closure of a set $C$, respectively. A point $x\in C$ is known as internal point of $C$ if for any $u$ there exists $t_\circ>0$ such that $x+tu\in C$ for any $|t|<t_\circ$ (see \cite[Definition 5.58]{aliprantis}). It is easy to notice that any interior point is internal point. Suppose $P:C\rightrightarrows X$ is a multi-valued map, that is, $P(x)\subseteq X$ for any $x\in C$. The graph of map $P$ is given as $Gr(P)=\{(x,z)\in C\times X \,|\,z\in P(x)\}$. 
	The reader may refer \cite{aliprantis} for some important concepts of upper semi-continuous (u.s.c.), lower semi-continuous (l.s.c.) and closed multi-valued maps. 
	
	Let us recall the classical definitions of variational and quasi-variational inequality problems \cite{stampbook}. Suppose $T:X\rightrightarrows X^*$ is a multi-valued map and $C\subseteq X$ is non-empty. Let $K:C\rightrightarrows C$ be a multi-valued map. Then, the quasi-variational inequality problem $QVI(T,K)$ corresponds to find a $\bar y\in K(\bar y)$ so that,
	\begin{equation*}
		\text{there exists}~\bar y^*\in T(\bar y)~\text{satisfying}~ \langle \bar y^*,y-\bar y\rangle\geq 0,~\text{for all}~y\in K(\bar y).
	\end{equation*}
	If the constraint map $K$ is constant, that is, $K(y)=C$ for all $y\in C$ then $QVI(T,K)$ reduces to a variational inequality problem $VI(T,C)$.
	
	

Let $X$ be a Banach space with topological dual $X^*$. A preference relation on $X$ is a binary relation denoted by $\succeq$. As per \cite[Chapter 1]{krepsweak}, $\succeq$ is known as, 
\begin{itemize}
	\item[-] \textit{reflexive} if $y\succeq y$ for any $y\in X$;
	\item[-] \textit{complete} if for any $y,z\in X$ we have $y\succeq z$ or $z\succeq y$ or both;
	\item[-] \textit{transitive} if for any $y,z,w \in X$ we have $y\succeq w$ and $w\succeq z$, then $y\succeq z$;
	\item[-] \textit{continuous} if the sets $\{z\in X\,|\, z\succeq y\}$ and $\{z\in X\,|\, y\succeq z\}$ are closed for any $y\in X$;
\end{itemize}
Any preference relation $\succeq$ defined on $X$ induces a strict preference relation $\succ$ as follows: For any pair $y,z\in X$, it is known that $y\succ z$ iff $y\succeq z$ and $z\nsucceq y$. A vector $\bar x\in K$ is known as maximal element of $K\subseteq X$ with respect to $\succeq$ if there no vector $y\in K$ such that $y\succ \bar x$. 
A real-valued function $u:X\rightarrow \mathbb{R}$ is known as numerical representation of $\succeq$ if for any pair $y,z\in X, y\succeq z$ iff $u(y)\geq u(z)$. If $X$ is connected and separable topological space and $\succeq$ is complete, transitive and continuous, then $\succeq$ admits a numerical representation (see \cite[Theorem I]{decision_processes}) and problem of finding maximal element reduces to classical optimization problem. On the other hand, a preference relation $\succeq$ admits a numerical representation only if it is complete and transitive \cite[Proposition 1.3]{krepsweak}. 
However, the completeness and transitivity of $\succeq$ are not feasible in real world situations (refer \cite{aumann}). Hence, we consider a non-ordered preference map $P:X\rightrightarrows X$ corresponding to preference relation $\succeq$ which need not be complete and transitive, 
\begin{equation}\label{preference}
	P(x)=\{y\in X|\,y\succeq x~\text{and}~ x\nsucceq y\}=\{y\in X|\,y\succ x\}.
\end{equation}
Now the notion of maximal elements can be given as (see \cite{donato,yannelis_infinite}),
\begin{definition}\label{maximal_element}
	Suppose $K$ is non-empty subset of Banach space $X$. Any $\bar x\in K$ is \textit{maximal} element for the preference map $P$ over the set $K$ if $P(\bar x)\cap K\neq \emptyset$. We denote the set of maximal elements for preference $P$ over set $K$ by  $\mathbb{S}(P,K)$. 
\end{definition}

\section{Mid-point continuity for Preference Maps}
We define the concept of mid-point continuity for non-ordered preference maps motivated by \cite{gorno}, which generalizes the concept of continuity for preference relations studied in \cite[Definition 2]{milasipref2021}:
\begin{definition}\label{midpoint}
	Let $X$ be a Banach space. A preference map $P:X\rightrightarrows X$ fulfils,
	\begin{itemize}
		\item[-] \textit{lower mid-point continuity} at any $x\in X$ if for any $w\in {P}(x)$, there exists $t \in [0,1)$ and open set $V$ in $X$ such that $t x+(1-t) w\in V$ and $V\subset {P}(x)$;
		\item[-] \textit{upper mid-point continuity} at any $x\in X$ if for any $w\in P(x)$, there exists $t \in [0,1)$ and open set $W$ in $X$ such that $t x+(1-t) w\in {P}(x')$ for each $x'\in W$;
		\item[-] \textit{mid-point continuity} \cite[Definition 1]{gorno} at any $x\in X$ if it is upper and lower mid-point continuous at $x$, that is, for any $w\in {P}(x)$, there exists $t \in [0,1)$ and open sets $V$ and $W$ in $X$ such that
		\begin{align*}
			t x+(1-t) w\in V, x\in W ~\text{and}~w'\in {P}(x')~\text{for all}~(w',x')\in V\times W.
		\end{align*}
	\end{itemize}
\end{definition}
Let $Y$ be a Banach space and $(x,y)\in X\times Y$ be arbitrary. We say $P:X\times Y\rightrightarrows X$ is lower mid-point continuous w.r.t. $x$ at $(x,y)$ if for any $w\in {P}(x,y)$, there exists $t \in [0,1)$ and open set $V$ in $X$ such that $t x+(1-t) w\in V$ and $V\subset {P}(x,y)$. Furthermore, we say $P$ is upper mid-point continuous w.r.t. $x$ at $(x,y)$ if for any $w\in {P}(x,y)$ there exists $t \in [0,1)$ and open sets $W$ in $X$ and $O$ in $Y$ such that  $t x+(1-t) w\in {P}(x',y')$ for each $(x',y')\in W\times O$. Finally, we say $P$ is mid-point continuous w.r.t. $X$ if it is mid-point continuous w.r.t. $x$ for all $(x,y)\in X\times Y$.

In \cite{milasipref2021,milasipref}, the maximization of preference is studied through variational inequalities under the continuity of preference relation $\succ$ (see \cite[Theorem 5]{milasipref2021}). The notion of mid-point continuity of preference maps is weaker than existing concept of continuity of preference relations $\succ$ studied in \cite[Definition 2.3]{milasipref} and \cite[Definition 2]{milasipref2021}. In fact, it is easy to observe that if preference relation $\succ$ is lower (or upper) semi-continuous over $X$ then $P:X\rightrightarrows X$ defined as (\ref{preference}) is lower (or upper) mid-point continuous, respectively. But, the converse need not be true as shown in following example.
\begin{example}\label{ex1}
	Suppose $\succ$ is defined on $[0,1]$ as $x\succ y$ iff $u(x)> u(y)$ where,
	\begin{equation}
		u(x)=\begin{cases}
			x,~&\text{if}~x<\frac{1}{2}\\
			2-x,~&\text{otherwise}.
		\end{cases}
	\end{equation}
	Then, clearly $\frac{1}{2}\succ 1$ but there no open set $V$ such that $\frac{1}{2}\in V$ and $x'\succ 1$ for each $x'\in V$. Hence, $\succ$ is not lower semi-continuous. But it is easy to observe that the corresponding preference map $P:[0,1]\rightrightarrows [0,1]$ given as,
	\begin{align*}
		P(x)=\begin{cases}
			(x,1],&~\text{if}~x\in [0,\frac{1}{2})\\
			\emptyset,&~\text{if}~x=\frac{1}{2}\\
			[\frac{1}{2},x),&~\text{if}~x\in (\frac{1}{2},1].
		\end{cases}
	\end{align*}
	is mid-point continuous at any $x\in [0,1]$. 
\end{example}

The generalized games are studied in finite dimensional spaces through variational reformulation in \cite{cotrina} under the assumption that preference maps are lower semi-continuous with open convex values. Furthermore, the maximization of preference maps is studied in finite dimensional spaces through variational reformulation in \cite{donato} under the lower semi-continuity and openness like assumptions on preference maps. Following result along with Example \ref{ex1} and \ref{ex2} shows that 
mid-point continuity is weaker than the assumptions in \cite{cotrina,donato}. 
\begin{proposition}\label{proplsc}
	Suppose $Z\subseteq \mathbb{R}^m$ is non-empty. A map $P:Z\rightrightarrows Z$ is:
	\begin{itemize}
		\item[(a)] lower mid-point continuous at $x$ if $P(x)$ is open in $Z$;
		\item[(b)] lower mid-point continuous at $x$ if $P(x)$ is convex and any $y\in P(x)$ is an internal point with respect to $Z$; 
		\item[(c)] upper mid-point continuous at $x$ if it is l.s.c. with convex and open values in $Z$;
		\item[(d)] upper mid-point continuous at $x$ if $P$ is l.s.c. at $x$, $P(x)$ is convex and any $y\in P(x)$ is an internal point with respect to $Z$.
	\end{itemize} 
\end{proposition}
\begin{proof}
	To prove (a), suppose $P$ is open valued in $Z$. Then, for any $x\in Z$ the lower mid-point continuity holds by assuming $t=0$. Further, we notice that $P(x)$ is open if it satisfies the assumptions given in (b) (see \cite[Lemma 5.60]{aliprantis}). Hence, lower mid-point continuity of $P$ follows by (a). For (c), we assume $P$ is l.s.c. with convex and open values in $Z$. Then, for any $x\in Z$ the upper mid-point continuity holds by employing \cite[Proposition 1]{zhou} and assuming $t=0$. Again, the proof of (d) follows from (c) and \cite[Lemma 5.60]{aliprantis}.
\end{proof}
\begin{remark}
	We observe that the converse of the statements in Proposition \ref{proplsc} need not be true. In fact, it is clear from Example \ref{ex1} that $P$ is lower mid-point continuous but it does not admit open values in $[0,1]$. Further, in following example $P$ is nor lower semi-continuous neither it admits open values but it is upper mid-point continuous for any $x\in [0,1]$.
\end{remark}
\begin{example} \label{ex2}
	Suppose $P:[0,1]\rightrightarrows[0,1]$ is defined as,
	\begin{align*}
		P(x)=\begin{cases}
			(x,1],&~\text{if}~x\in [0,\frac{1}{2}]\\
			[\frac{1}{2},\frac{3}{4}],&~\text{if}~x\in (\frac{1}{2},1].
		\end{cases}
	\end{align*}
	Clearly, $P$ is not l.s.c. at $x=\frac{1}{2}$ and it is not even open valued but it is upper mid-point continuous for any $x\in [0,1]$.
\end{example}

\section{Normal Cone Operators for Preference Maps}
In this section, we derive some important properties of normal cone operators corresponding to non-ordered preference map, which will be later employed to study preference maximization problem and generalized games through variational reformulation.

According to \cite{ausselnormal}, the normal cone of the set $C$ at some point $x\in X$ is
\begin{align}
	N_C(x)=
	\begin{cases}
		\{x^*\in X^*|\,\langle x^*,y-x\rangle \leq 0~\text{for all}~y\in C\},&\text{if}~C\neq\emptyset\\
		X^*,&\text{otherwise}.
	\end{cases}
\end{align}


For Banach spaces $X$ and $Y$, consider a set-valued map $P:X\times Y\rightrightarrows X$. We define a normal cone operator $N_{P}:X\times Y\rightrightarrows X^*$ as
$N_{P}(x,y)= N_{P(x,y)} (x)=(P(x,y)-\{x\})^\circ,$ that is, 
\begin{align}\label{normal}
	N_{P}(x,y)=
	\begin{cases}
		\{x^*|\,\langle x^*,z-x\rangle \leq 0~\forall\,z\in P(x,y)\},&\text{if}~{P}(x,y)\neq\emptyset\\
		X^*,&\text{otherwise}.		
	\end{cases}	
\end{align}



We derive some properties related to normal cone operators defined as (\ref{normal}). 
\begin{proposition}\label{nonempty}
	Suppose $N_P:X\times Y\rightrightarrows X^*$ is defined as (\ref{normal}). If the map $P$ is convex valued and lower mid-point continuous w.r.t. $x$ at $(x,y)$ and $x\notin P(x,y)$ then $(N_P(x,y)\setminus \{0\})\neq \emptyset$ for any $(x,y)\in X\times Y$.
\end{proposition}
\begin{proof}
	In the case $P(x,y)=\emptyset$ then we have $N_P(x,y)=X^*$ and $N_P(x,y)\setminus \{0\}\neq \emptyset$. Suppose $P(x,y)\neq \emptyset$ and $w\in P(x,y)$. Then, by lower mid-point continuity of $P$ there exists $t\in [0,1)$ and open set $V$ such that $tx+(1-t)w\in V\subset P(x,y)$. This shows $P(x,y)$ contains an interior point $tx+(1-t)w$. As $x\notin P(x,y)$, we obtain $0\neq x^*\in X^*$ by virtue of \cite[Theorem 5.67]{aliprantis} such that
	\begin{equation}
		\sup_{w\in P(x)}\langle x^*,w\rangle \leq \langle x^*,x\rangle.
	\end{equation}
	This proves $N_P(x,y)\setminus \{0\}\neq \emptyset$.
\end{proof}
\begin{proposition}\label{intersection}
	Suppose $N_P:X\times Y\rightrightarrows X^*$ is defined as (\ref{normal}). If $P$ is lower mid-point continuous w.r.t. $x$ at $(x,y)$ and $P(x,y)\neq \emptyset$ then,
	\begin{equation}
		N_{P}(x,y)\cap -N_{P}(x,y)=\{0\}.
	\end{equation}
\end{proposition}
\begin{proof}
	We observe that, $${N}_{P}(x,y)\cap -{N}_{P}(x,y)=(P(x,y)-\{x\})^\perp=\{x^*|\,\langle x^*,w-x\rangle =0,~\forall~w\in P(x,y)\}.$$ 
	Suppose $w\in P(x,y)$, then we have $t \in [0,1)$ and open set $V$ such that $t x+(1-t) w\in V\subset P(x,y)$ by lower mid-point continuity. This implies $int (P(x,y)-\{x\})\neq \emptyset$. Thus, $(P(x,y)-\{x\})^\perp=\{0\}$.
\end{proof}
\begin{proposition} \label{Tusc}
	Suppose $P:X\times Y\rightrightarrows X$ is upper mid-point continuous map with respect to $x$ at $(x,y)$. 
	For $C\subset X\times Y$, let $T:C\rightrightarrows X^*$ be defined as $T(x,y)=N_P(x,y)\cap H$, where $H$ is weak$^*$ closed set in $X^*$. Then, the map $T$ is norm-weak$^*$ closed at $(x,y)$ if $T(C)\subset B_X^*$.
\end{proposition}
\begin{proof}
	Consider a net $(x_\beta,y_\beta, x_\beta^*)_{\beta}\subseteq Gr(T)$ such that $(x_\beta,y_\beta)_\beta$ converges to $(x,y)$ with respect to norm topology and $x_\beta^*$ converges to $x^*$ with respect to weak$^*$ topology. We aim to show that $x^*\in T(x,y)$. Since $(x_\beta^*)_{\beta}\subseteq H$ weak$^*$ converges to $x^*$ and $H$ is weak$^*$ closed set, we have $x^*\in H$. It remains to show that $x^*\in N_P(x,y)$, that is,
	\begin{equation}\label{normal_lsc}
		\langle x^*,z-x\rangle\leq 0,~\text{for each}~z\in P(x,y).
	\end{equation}
	As we know $x_\beta^*\in  N_P(x_\beta)$, it appears
	\begin{equation}\label{normal_lsc_2}
		\langle x_\beta^*,z-x_\beta \rangle\leq 0~\text{for all}~z\in P(x_\beta,y_\beta).
	\end{equation} 
	To prove (\ref{normal_lsc}), we assume $z\in P(x,y)$ is arbitrary. Since $P:X\times Y\rightrightarrows X$ is upper mid-point continuous w.r.t. $x$ at $(x,y)$, we obtain $t\in [0,1)$ and a subnet $(x_\gamma,y_\gamma)_{\gamma}\subseteq (x_\beta,y_\beta)_{\beta}$ such that $tx+(1-t) z\in P(x_\gamma,y_\gamma)$ for each $\gamma$. As per (\ref{normal_lsc_2}), we have
	\begin{equation}\label{normal_lsc3}
		\langle x_\gamma^*,tx+(1-t) z-x_\gamma \rangle\leq 0.
	\end{equation}
	Since the evaluation $\langle\cdot,\cdot\rangle$ restricted to $B_X^*\times X$ is jointly continuous and $(x^*_\gamma)_\gamma\subseteq B_X^*$, we observe (\ref{normal_lsc}) holds by taking limits in (\ref{normal_lsc3}).
\end{proof}

In \cite[Theorem 2.5]{giuli} Castellani-Giuli derived the presence of norm-weak$^*$ u.s.c. compact convex valued sub-map for an adjusted normal operator corresponding to lower semi-continuous real valued function. In following result, we show that a sub-map with similar properties exists for normal cone operators corresponding to non-ordered preference maps motivated by the approach in \cite{giuli}.
\begin{proposition}\label{normalmap}
	Suppose $P:X\times Y\rightrightarrows X$ is convex valued mid-point continuous map w.r.t. to $X$ and $x\notin P(x,y)$ for any $(x,y)\in X\times Y$. Then, there exists a norm-weak$^*$ upper semi-continuous map $F:X\times Y\rightrightarrows X^*$ with non-empty convex weak$^*$ compact values such that $F(x,y)\subseteq N_P(x,y)\setminus \{0\}$ whenever $P(x,y)\neq \emptyset$.
\end{proposition}
\begin{proof}
	Suppose $(\bar x,\bar y)\in (X\times Y)\setminus E$ is arbitrary, where $E=\{(x,y)\in X\times Y|\,P(x,y)=\emptyset\}$. We aim to find an open set $O_{(\bar x,\bar y)}$ in $(X\times Y)\setminus E$ containing $(\bar x,\bar y)$ and a map $T_{(\bar x,\bar y)}:O_{(\bar x,\bar y)}\rightrightarrows X^*$ such that $T_{(\bar x,\bar y)}(x,y)\subseteq N_P(x,y)\setminus\{0\}$ for all $(x,y)\in O_{(\bar x,\bar y)}$. Suppose $\bar w\in P(\bar x,\bar y)$ then by mid-point continuity of $P$ w.r.t. $\bar x$ there exists $\epsilon_1,\epsilon_2>0$ and $t=t_{(\bar x,\bar y)}\in[0,1)$, such that 
	\begin{equation}
		t \bar x+(1-t)\bar w+\epsilon_1 B_X\subset P(x,y)~\text{for each}~(x,y)\in (\bar x,\bar y)+\epsilon_2 (B_{X}\times B_{Y}).
	\end{equation}
	Suppose $2\epsilon=\min\{\epsilon_1,\epsilon_2\}$ then we have
	\begin{equation}
		t\bar x+(1-t)\bar w+2\epsilon B_X\subset P(x,y)~\text{for each}~(x,y)\in (\bar x,\bar y)+\epsilon (B_{X}\times B_{Y}).
	\end{equation}
	For any $(x,y)\in (\bar x,\bar y)+\epsilon (B_{X}\times B_{Y})$ and $x^*\in N_P(x,y)=\{x^*|\,\langle x^*,w-x\rangle\leq 0,~\text{for any}~w\in P(x,y)\}$ we have,
	\begin{equation}
		\langle x^*, t \bar x+(1-t)\bar w+2\epsilon u-x\rangle\leq 0~\text{for all}~u\in B_X.
	\end{equation}
	This further implies,
	\begin{align*}
		2 \epsilon\norm{x^*}_*&= 2 \epsilon\sup_{u\in B_X} \langle x^*,u\rangle\\
		&\leq \langle x^*,x-(t_z \bar x+(1-t_z)w_z)\rangle\\
		&= \langle x^*, x-\bar x\rangle+(1-t_z)\langle x^*,\bar x-\bar w\rangle\\
		&\leq  \epsilon\norm{x^*}+(1-t_z) \langle x^*,\bar x-\bar w\rangle.
	\end{align*}
	Finally, we have $\epsilon\norm{x^*}_*\leq (1-t)\langle x^*,\bar x-\bar w\rangle$.
	Let us assume,
	$$H_{(\bar x,\bar y)}=\{x^*\in X^*|\, (1-t)\langle x^*,\bar x-\bar w\rangle =\epsilon\}.$$ Clearly, $N_P(x,y)\cap H_{z}\subset N_P(x,y)\setminus\{0\}$ is weak$^*$ compact as it is contained in $B_X^*$ and non-empty as per Proposition \ref{nonempty}. Suppose $O_{(\bar x,\bar y)}:= (\bar x,\bar y)+\epsilon (B_X\times B_Y)$ and define $T_{(\bar x,\bar y)}:O_{(\bar x,\bar y)}\rightrightarrows X^*$ by,
	\begin{equation}
		T_{(\bar x,\bar y)}(x,y)=N_P(x,y)\cap H_{(\bar x,\bar y)}~\text{for any}~(x,y)\in O_{(\bar x,\bar y)}.
	\end{equation}
	Now, we construct a map $T:X\times Y\rightrightarrows X^*$ as convex combination of the maps like $T_{(\bar x,\bar y)}$ by employing a partition of unity technique. We observe that, $$\mathcal{O}=\{O_{(\bar x,\bar y)}:= {(\bar x,\bar y)}+\epsilon (B_X\times B_Y) |\,{(\bar x,\bar y)}\in (X\times Y)\setminus E\}$$ forms an open cover for $(X\times Y)\setminus E$. Since $X\times Y$ is paracompact, there exists a locally finite open refinement $\mathcal{U}=\{U_\nu\}_{\nu\in \Lambda}$ for $\mathcal{O}$ such that for each $\nu\in \Lambda$ there is $(\bar x,\bar y)\in X\times Y$ satisfying $U_\nu\subset O_{(\bar x,\bar y)}$. Further, we obtain a family of continuous functions $\{f_\nu\}_{\nu\in\Lambda}$ which forms a partition of unity corresponding to $\mathcal{U}$ as per \cite[Theorem 2.90]{aliprantis}. Then, $f_\nu(x,y)=0$ if $(x,y)\notin U_\nu$ and $\sum_{\nu\in \Lambda} f_\nu(x,y)=1$ for any $(x,y)\in X\times Y$. Further, the set $\lambda(x,y)=\{\nu\in \Lambda|\,f_\nu(x,y)>0\}$ is non-empty finite. Since for each $\nu\in \Lambda$ there exists $(\bar x,\bar y)\in X$ satisfying $U_\nu\subset O_{(\bar x,\bar y)}$, we define $T_\nu:U_\nu\rightrightarrows X^*$ by $T_\nu(x,y)=T_{(\bar x,\bar y)}(x,y)$.
	
	Define $T:(X\times Y)\setminus E\rightrightarrows X^*$ as $T(x,y)=\sum_{\nu\in \Lambda(x,y)} f_\nu(x,y) T_\nu(x,y)$. We claim that $T(x,y)\subset N_P(x,y)\setminus \{0\}$.
	On contrary, suppose $0\in T(x,y)$. Then, there exists $x_\nu^*\in T_{\nu} (x,y)\subset N_P(x,y)$ for all $\nu\in \Lambda(x,y)$ such that
	$0=\sum_{\nu\in \Lambda(x,y)} f_\nu(x,y) x_\nu^*$. Since $f_{\nu}(x,y)>0$ for all $\nu\in \Lambda(x,y)$, we have
	\begin{equation*}
		-x_\nu^*=\sum_{\mu\in \Lambda(x,y)\setminus \{\nu\}} \frac{f_\mu(x,y)}{f_{\nu}(x,y)} x_\mu^*\in N_P(x,y).
	\end{equation*}
	In the view of Proposition \ref{intersection}, we have $x_\nu^*=\{0\}$. But, this contradicts the fact that $x_\nu^*\in H_z$. This proves $T(x,y)\subset N_P(x,y)\setminus \{0\}$.
	
	We claim that $F:X\times Y\rightrightarrows X^*$ constructed as,\begin{equation} \label{normalprop_eq3}
		F(x,y)=\begin{cases}
			B_X^*,&~\text{if}~P(x,y)=\emptyset\\
			T(x,y),&~\text{otherwise}
		\end{cases}
	\end{equation}
	is norm-weak$^*$ upper semi-continuous. Since $F(x,y)\subseteq B_X^*$, it is sufficient to prove that $F$ is norm-weak$^*$ closed. Consider a net $(x_\beta,y_\beta, x_\beta^*)_{\beta}\subseteq Gr(F)$ such that $(x_\beta,y_\beta)_\beta$ converges to $(x,y)$ with respect to norm topology and $x_\beta^*$ converges to $x^*$ with respect to weak$^*$ topology. We aim to show that $x^*\in F(x,y)$. If $(x,y)\in E$, we observe that $x^*\in F(x,y)=B_X^*$ as $x_\beta^*\in F(x_\beta,y_\beta)\subset B_X^*$ and $B_X^*$ is weak$^*$ closed. Suppose $(x,y)\in (X\times Y)\setminus E$. We notice that $E$ is closed set due to mid-point continuity of $P$. Hence, it is not restrictive to consider $(x_\beta,y_\beta)_\beta \subset (X\times Y)\setminus E$. Since $(x_\beta,y_\beta)_\beta$ converges to $(x,y)$, we can assume without loss of generality that $\Lambda(x,y)\subseteq \Lambda(x_\beta,y_\beta)$ for each $\beta$. Hence, it appears,
	\begin{equation*}
		T(x_\beta,y_\beta)=\sum_{\nu\in \Lambda(x_\beta,y_\beta)\setminus \Lambda(x,y)} f_\nu(x_\beta,y_\beta) T_\nu(x_\beta,y_\beta)+ \sum_{\nu\in \Lambda(x,y)} f_\nu(x_\beta,y_\beta) T_\nu(x_\beta,y_\beta).
	\end{equation*}
	Since $(x^*_\beta)_\beta$ weak$^*$ converges to $x^*$, we have nets $(x^*_{\nu,\beta})_\beta$ weak$^*$ converging to $x^*_\nu$ for each $\nu\in \Lambda(x_\beta,y_\beta)$ such that,
	\begin{equation}\label{normalprop_eq4}
		x^*_\beta=\sum_{\nu\in \Lambda(x_\beta,y_\beta)\setminus \Lambda(x,y)} f_\nu(x_\beta,y_\beta) x^*_{\nu,\beta}+ \sum_{\nu\in \Lambda(x,y)} f_\nu(x_\beta,y_\beta) x^*_{\nu,\beta}.
	\end{equation}
	Since $\{f_\nu\}_{\nu\in \Lambda}$ forms a continuous family of functions we have,
	\begin{equation*}
		\lim_{\beta} \sum_{\nu\in \Lambda(x_\beta,y_\beta)\setminus \Lambda(x,y)} f_\nu(x_\beta,y_\beta) = 1-\lim_{\beta}\sum_{\nu\in \Lambda(x,y)} f_\nu(x_\beta,y_\beta)=0.
	\end{equation*}
	This further implies first term in (\ref{normalprop_eq4}) weak$^*$ converges to $0$ as,
	\begin{equation*}
		\norm{\sum_{\nu\in \Lambda(x_\beta,y_\beta)\setminus \Lambda(x,y)} f_\nu(x_\beta,y_\beta) x^*_{\nu,\beta}}_*\leq \sum_{\nu\in \Lambda(x_\beta,y_\beta)\setminus \Lambda(x,y)} f_\nu(x_\beta,y_\beta).
	\end{equation*}
	In the view of Proposition \ref{Tusc}, we observe that each $T_\nu$ is norm-weak$^*$ closed map and $x^*_\nu\in T_\nu(x,y)$. Hence, by taking limits in (\ref{normalprop_eq4}) we have,
	\begin{equation*}
		x^*= \sum_{\nu\in \Lambda(x,y)} f_\nu(x,y) x^*_\nu\in T(x,y).
	\end{equation*}
	This proves the map $F$ defined as (\ref{normalprop_eq3}) is norm-weak$^*$ u.s.c. with non-empty convex weak$^*$ compact values such that $F(x,y)\subseteq N_P(x,y)$.
\end{proof}
\begin{remark}\label{infinite_remark}
	In the case $X$ and $Y$ are finite dimensional spaces, the map $F$ defined as $F(x,y)=co(N_P(x,y)\cap S^*_X)$, where $S^*_X=\{y\in X^*\,|\,\norm{y}_*=1\}$, is u.s.c. with non-empty, convex and compact values as per \cite{donato}. However, the similar properties does not hold for $F$ if $X$ is infinite dimensional space, as it is well known that the sphere $S^*_X$ is not weak$^*$ compact in $X^*$.
\end{remark}

\begin{proposition}\label{relation}
	Suppose $P$ is convex valued lower mid-point continuous map w.r.t. $x$ at $(x,y)$. If $x^*\in N_{P}(x,y)$ and $w\in P(x,y)$ satisfy $\langle x^*,w-x\rangle\geq 0$ then $x^*=0$. 
\end{proposition}
\begin{proof}
	Let $C$ be non-empty subset of $X$. We claim that if $y\in int(C)$ and $y^*\in C^\circ$ satisfy $\langle y^*,y \rangle\geq 0$ then $y^*=0$. Since $y\in int(C)$, we have $\epsilon>0$ such that $y+\epsilon B_X\subset C$. Hence,
	\begin{equation}\label{prop_normal}
		\langle y^*,y+\epsilon u\rangle \leq 0~\text{for all}~u\in B_X.
	\end{equation}
	Since $y^*\in C^\circ$, we have $\langle y^*,y\rangle=0$. From (\ref{prop_normal}), we obtain $\langle y^*,u\rangle \leq 0~\text{for all}~u\in B_X.$ Thus, $\langle y^*,u\rangle=0~\text{for all}~u\in B_X,$ which proves $y^*=0$.
	
	Suppose $x^*\in N_{P}(x,y)$ and $w\in P(x,y)$ satisfy $\langle x^*,w-x\rangle\geq 0$. Considering $w_t=tx+(1-t)w$, we have
	\begin{equation}\label{prop_normal2}
		\langle x^*,w_t-x\rangle=(1-t)\langle x^*,w-x\rangle\geq 0~\text{for any}~t\in [0,1).
	\end{equation}
	By lower mid-point continuity, we get $t \in [0,1)$ such that $w_t \in int(P(x,y))$. In the view of (\ref{prop_normal2}) and the fact that $w_t-x\in int(P(x,y)-\{x\})$ and $x^*\in (P(x,y)-\{x\})^\circ$ we get $x^*=0$. 
\end{proof}

\section{Existence Results}\label{sec_existence_results}
\subsection{Variational Reformulation of Preference Maximization Problems}
Let us reconsider the problem of finding maximal elements for non-ordered preference map $P:X\rightrightarrows X$ over $K\subset X$, by following the Definition \ref{maximal_element}. We show that any solution of variational inequality with principal operator $F:X\rightrightarrows X^*$ defined as (\ref{normalprop_eq3}) is maximal element for $P$ over $K$. Note that the results derived in the previous section for normal cone operators can be easily adapted for preference maps $P$ which have no dependence on $y\in Y$.
\begin{theorem}\label{reformulation}
	Let $K$ be non-empty closed convex subset of Banach space $X$. Suppose $P:X\rightrightarrows X$ is convex valued mid-point continuous map and $x\notin P(x)$ for any $x\in X$. Then, $\bar x\in \mathbb{S}(P,K)$ if $\bar x\in K$ solves $VI(F,K)$ where $F$ is defined as (\ref{normalprop_eq3}). 
\end{theorem}
\begin{proof}
	Suppose $\bar x\in K$ solves $VI(F,K)$, that is
	\begin{equation}\label{thmeq1}
		\exists\,\bar x^*\in F(\bar x), \langle \bar x^*,y-\bar x\rangle\geq 0~\text{for all}~y\in K.
	\end{equation}
	If $P(\bar x)=\emptyset$ then $\bar x$ is maximal element. Suppose $P(\bar x)\neq \emptyset$, then we aim to show that $P(\bar x)\cap K=\emptyset$. On contrary, suppose $z\in P(\bar x)\cap K$. From (\ref{thmeq1}), it appears that 
	$\bar x^*\in F(\bar x)\subseteq N_P(\bar x)\setminus\{0\}$ satisfies,
	\begin{equation}
		\langle \bar x^*,z-\bar x\rangle \geq 0.
	\end{equation} 
	According to Proposition \ref{relation}, we have $\bar x^*=0$. But, we know that $P(\bar x)\neq \emptyset$ and $x^*\in F(\bar x)\subset N_P(\bar x)\setminus \{0\}$ by following the arguments in proof of Proposition \ref{normalmap}. This shows our assumption $P(\bar x)\cap K\neq \emptyset$ is false and $\bar x\in \mathbb{S}(P,K)$.
\end{proof}
Based on the variational reformulation of preference maximization problems obtained in the above result Theorem \ref{reformulation}, we now deduce the sufficient conditions for the existence of maximal elements.
\begin{theorem}\label{existence_maximal}
	Suppose $P:X\rightrightarrows X$ is convex valued mid-point continuous map satisfying $x\notin P(x)$ for any $x\in X$ and $K\subset X$ is non-empty compact and convex. Then, the set of maximal elements $\mathbb{S}(P,K)\neq \emptyset$.
\end{theorem}
\begin{proof}
	In the view of Theorem \ref{reformulation}, it is sufficient to show that $VI(F,K)$ admits a solution. According to Proposition \ref{normalmap}, the map $F$ is norm-weak$^*$ upper semi-continuous with non-empty convex and weak$^*$ compact values. Hence, we see that variational inequality $VI(F,K)$ admits a solution $\bar x\in K$ by assuming the constraint map $K$ is a constant for all $x$ in \cite[Theorem 3.3]{giuli}. Finally, $\bar x\in K$ is required maximal element of $P$ over $K$.
\end{proof}

\subsection{Variational Reformulation of Generalized Games}
In this section, we aim to show that any solution of quasi-variational inequality with principal operator $\mathcal{F}$ (defined as (\ref{mathcal_F})) and constraint map $K=\prod_{\nu\in \Lambda} K_\nu$ is an equilibrium for game $\Gamma=(X_\nu,K_\nu,P_\nu)_{\nu\in \Lambda}$ (refer (\ref{GNEPshafer}) in Section \ref{intro}).

By following (\ref{normal}), we define ${N}_{P_\nu} (x):X_\nu\times X_{-\nu}\rightrightarrows X^*_\nu$ corresponding to preference $P_\nu$ as $N_{P_\nu}(x_\nu,x_{-\nu})=N_{P_\nu (x)} (x_\nu)$, that is,
\begin{equation*}
	N_{P_\nu}(x_\nu,x_{-\nu})= \begin{cases}
		\{x_\nu^*\in X_\nu^*|\,\langle x_\nu^*,y_\nu-x_\nu\rangle \leq 0~\text{for all}~y_\nu\in P_\nu (x)\},&\text{if}~P_\nu (x)\neq\emptyset\\
		X_\nu^*,&\text{otherwise}.
	\end{cases}
\end{equation*}
In addition to notations in (\ref{X}), let $X^*=\prod_{\nu\in \Lambda}X_\nu^*$. Suppose $\mathcal{F}:X\rightrightarrows X^*$ is defined as,
\begin{equation}\label{mathcal_F}
	\mathcal{F}(x)=\prod_{\nu\in \Lambda} F_\nu(x),
\end{equation} 
where $F_\nu:X_\nu\times X_{-\nu}\rightrightarrows X_\nu^*$ is considered as (\ref{normalprop_eq3}).

Following result shows that any solution of $QVI(\mathcal{F},K)$ is equilibrium for the generalized game $\Gamma=(X_\nu,K_\nu,P_\nu)_{\nu\in \Lambda}$ considered in Section \ref{intro}.
\begin{theorem}\label{reformulation_game}
	For any $\nu\in \Lambda$, assume that,
	\begin{itemize}
		\item[(a)] $P_\nu:X\rightrightarrows X_\nu$ is convex valued mid-point continuous map with respect to $X_\nu$ such that $x_\nu\notin P(x)$ for any $x\in X$; 
		\item[(b)] $K_\nu:C\rightrightarrows C_\nu$ admits non-empty closed convex values.
	\end{itemize} 
	Then, $\bar x \in \mathbb{S}(\Gamma)$ if $\bar x$ solves $QVI(\mathcal{F},K)$ where $\mathcal{F}$ is defined as (\ref{mathcal_F}).
\end{theorem}	
\begin{proof}
	If $\bar x$ solves $QVI(\mathcal{F},K)$, then 
	\begin{equation}\label{thm_eq1}
		\exists\, \bar x^*\in \mathcal{F}(\bar x), \langle \bar x^*,y-\bar x\rangle \geq 0, ~\text{for all}~y\in K(\bar x).
	\end{equation}
	We claim that $\bar x\in \mathbb{S}(\Gamma).$ Suppose $\nu\in \Lambda$ is arbitrary. If $P_\nu (\bar x)=\emptyset$ then $P_\nu(\bar x)\cap K_\nu (\bar x)=\emptyset$. In the case $P_\nu(\bar x)\neq \emptyset$, we aim to show that $P_\nu(\bar x)\cap K_\nu(\bar x)=\emptyset$. On contrary, suppose $z_\nu\in P_\nu(\bar x)\cap K_\nu(\bar x)$. By substituting $y=(z_\nu,\bar{x}_{-\nu})$ in  
	(\ref{thm_eq1}) it appears that $\bar x_\nu^*\in F_\nu (\bar x_\nu,\bar x_{-\nu})\subset N_{P_\nu} (\bar x_\nu,\bar x_{-\nu})$ satisfies,
	\begin{equation}
		\langle \bar x_\nu^*,z_\nu-\bar x_\nu\rangle \geq 0.
	\end{equation} 
	According to from Proposition \ref{relation}, we have $\bar x_\nu^*=0$. But, 
	we know that $P_\nu(\bar x)\neq \emptyset$ and $\bar x_\nu^*\in F_\nu (\bar x)\subseteq N_{P_\nu} (\bar x)\setminus \{0\}$ by following the arguments given in proof of Proposition \ref{normalmap}. This shows, our assumption $P_\nu(\bar x)\cap K_\nu(\bar x)\neq \emptyset$ is false. Since $\nu$ is chosen arbitrarily, we observe $P_\nu(\bar x)\cap K_\nu(\bar x)= \emptyset$ for any $\nu\in \Lambda$ and $\bar x\in \mathbb{S}(\Gamma)$. 
\end{proof}
Based on the variational reformulation of generalized games obtained in the above result Theorem \ref{reformulation_game}, we now deduce the sufficient conditions for the existence of equilibrium.
\begin{theorem}\label{existence}
	For any $\nu\in \Lambda$, assume that,
	\begin{itemize}
		\item[(a)] $P_\nu:X\rightrightarrows X_\nu$ is convex valued mid-point continuous map with respect to $X_\nu$ such that $x_\nu\notin P(x)$ for any $x\in X$; 
		\item[(b)] $K_\nu:C\rightrightarrows C_\nu$ is closed lower semi-continuous map with non-empty convex values and $K_\nu (C)$ relatively compact. 
	\end{itemize} 
	Then, the set of equilibrium points $\mathbb{S}(\Gamma)\neq \emptyset$.
\end{theorem}
\begin{proof}
	In the view of Theorem \ref{reformulation_game}, it is sufficient to prove that $QVI(\mathcal{F},K)$ admits a solution. As per Proposition \ref{normalmap}, we know that the map $\mathcal{F}=\prod_{\nu}F_\nu$ is norm-weak$^*$ upper semi-continuous with non-empty convex and weak$^*$ compact values. Finally, we observe that $QVI(\mathcal{F},K)$ admits a solution $\bar x\in K(\bar x)$ as per \cite[Theorem 3.3]{giuli}. Finally, $\bar x\in \mathbb{S}(\Gamma)$ is the required equilibrium for generalized game $\Gamma$. 
\end{proof}



\begin{thebibliography}{99}
	\bibitem{aliprantis} C.D. Aliprantis, K.C. Border, Infinite Dimensional Analysis: A Hitchhiker's Guide, Springer, Berlin, 2006.
	\bibitem{debreu} K.J. Arrow, G. Debreu, Existence of an equilibrium for a competitive economy, Econometrica 22 (1954) 265--290.
	
		\bibitem{aumann} R.J. Aumann: Utility theory without the completeness axiom. Econometrica 30 (1962) 445--462.
	
	\bibitem{ausselnormal} D. Aussel, New developments in quasiconvex optimization, in: S.A.R. Al-Mezel, F.R.M. Al-Solamy, Q.H. Ansari (Eds.), Fixed Point Theory, Variational Analysis and Optimization, Chapman and Hall, CRC Press, 2014. 
	\bibitem{ausseltime} D. Aussel, R. Gupta, A. Mehra, Evolutionary variational inequality formulation of the generalized Nash equilibrium problem, J. Optim. Theory Appl. 169 (2016) 74--90.
	\bibitem{cotrina} O. Bueno, Y. Garc{\'\i}a, J. Cotrina, Generalized ordinal Nash games: Variational approach, Oper. Res. Letters  51 (2023) 353--356.
	\bibitem{giuli} M. Castellani, M. Giuli, A continuity result for the adjusted normal cone operator, J. Optim. Theory Appl. 200 (2024) 858--873.
	\bibitem{decision_processes} G. Debreu, Representation of a preference ordering by a numerical function, in: M.R. Thrall, C.H. Coombs, R.C. Davis (Eds.), Decision processes, 159--165, Wiley New York, 1954.
	\bibitem{donato} M.B. Donato, A. Villanacci, Variational inequalities, maximal elements and economic equilibria, J. Math. Anal. Appl. 519 (2023) 126769.
	
	
	\bibitem{faccsurvey} F. Facchinei, C. Kanzow, Generalized Nash equilibrium problems, Ann. Oper. Res. 175 (2010) 177--211.
	\bibitem{faraci} F. Faraci, F. Raciti, On generalized Nash equilibrium in infinite dimension: the Lagrange multipliers approach,
	Optimization, 64 (2012) 321--338.
	
	\bibitem{gorno} L. Gorno, A.T. Rivello, A maximum theorem for incomplete preferences, J. Math. Econ. 106 (2023) 102822.
	
	\bibitem{stampbook} D. Kinderlehrer, G. Stampacchia, An introduction to variational inequalities and their applications, SIAM, 2000.
	\bibitem{krepsweak} D.M. Kreps, Microeconomic foundations (Vol. 1): Choice and Competitive Markets, Princeton University Press, Princeton, 2013.
	\bibitem{milasipref} M. Milasi, A. Puglis, C. Vitanza, On the study of the economic equilibrium problem through preference relations, J. Math. Anal. Appl. 477 (2019) 153--162. 
	\bibitem{milasipref2021} M. Milasi, D. Scopelliti, A variational approach to the maximization of preferences without numerical representation, J. Optim. Theory Appl. 190 (2021) 879--893.
	\bibitem{shafer} W. Shafer, H. Sonnenschein, Equilibrium in abstract economies without ordered preferences, J. Math. Econ. 2 (1975) 345--348.  
	
	
	
	\bibitem{shivani1} A. Sultana, S. Valecha, Variational Reformulation of Generalized Nash Equilibrium Problems with Non-ordered Preferences, arXiv preprint (2023) 
	https://doi.org/10.48550/arXiv.2302.08702.  
	
	\bibitem{shivani2} A. Sultana, S. Valecha, Projected solution for Generalized Nash Games with Non-ordered Preferences,  arXiv preprint (2023) https://doi.org/10.48550/arXiv.2305.07275.
	
	
	
	\bibitem{yannelis_infinite} N. C. Yannelis, N. D. Prabhakar, Existence of maximal elements and equilibria in linear topological spaces, J. Math. Econ. 12 (1983) 223--245.
	
	\bibitem{zhou} J.X. Zhou, On the existence of equilibrium for abstract economies, J. Math. Anal. Appl. 193 (1995) 839-858.
	
\end{thebibliography}



	
	
	
\end{document}